\documentclass[]{article}
\usepackage{preamble, hypercustom,fullthm, float,subcaption,thmtools}

\theoremstyle{plain}

\title{Towers and elementary embeddings in toral relatively hyperbolic groups}

\author{Christopher Perez}

\begin{document}
\maketitle

\begin{abstract}
In a remarkable series of papers, Zlil Sela classified the first-order theories of free groups and torsion-free hyperbolic groups using geometric structures he called towers. It was later proved by Chlo\'e Perin that if $H$ is an elementarily embedded subgroup (or elementary submodel) of a torsion-free hyperbolic group $G$, then $G$ is a tower over $H$. We prove a generalization of Perin's result to toral relatively hyperbolic groups using JSJ and shortening techniques.
\end{abstract}


\section{Introduction}

Tarski's problem concerns the \emph{elementary theory} $\Th(G)$ of a group $G$, the set of all first-order sentences in the language of groups which are valid over $G$. The problem asks if all finitely generated, non-abelian free groups have the same elementary theory. This question was answered in the affirmative in 2006 by Zlil Sela \cite{Sela06b}, and independently by Olga Kharlampovich and Alexei Myasnikov \cite{KM06}. Sela went on to generalize the techniques he used to solve Tarski's problem and prove the following:
\begin{thm}[{\cite[Theorem 7.10]{Sela09}}]
Let $\Gamma$ be a torsion-free hyperbolic group, and let $G$ be a finitely generated group. If $\Th(G)=\Th(\Gamma)$, then $G$ is a hyperbolic group.
\end{thm}
This result was later generalized by Simon Andr\'e to apply to hyperbolic groups with torsion, as well, making hyperbolicity a first-order invariant among finitely generated groups \cite{Andre18}.

A concept closely related to elementary equivalence is that of elementary embeddings: Let $H$ be a subgroup of a group $G$. The inclusion of $H$ into $G$ is an \emph{elementary embedding}, denoted $H\elemb G$, if for any first-order formula $\ph(x_1,\ldots,x_n)$ and any $(h_1,\ldots,h_n)\in H^n$ with $n\geq 0$,
\[H\models\ph(h_1,\ldots,h_n)\iff G\models\ph(h_1,\ldots,h_n),\]
where $\ph$ is a first-order sentence when $n=0$ and $G\models\ph(g_1,\ldots,g_n)$ denotes that $G$ \emph{models} $\ph(x_1,\ldots,x_n)$ with the assignment $x_i=g_i$, i.e., $\ph(g_1,\ldots,g_n)$ is true over $G$. In particular, $\Th(H)=\Th(G)$ if $H$ is elementarily embedded.

\emph{Towers} first appeared in Sela's work, and roughly speaking, towers are built from \emph{floors} in which a group retracts onto its \emph{base} in a nice way, and one can find these retractions if there exists a \emph{preretraction}. In this paper we using a version of towers and floors due to Vincent Guirardel, Gilbert Levitt, and Rizos Sklinos which depend \emph{centered} and \emph{retractable splittings} \cite{GLS18}. Towers, splittings, and preretractions are discussed in more detail in \S\ref{tap}.

\begin{figure}[t]\label{tower}
\begin{center}
\begin{subfigure}[t]{0.46\textwidth}\begin{center}
\includegraphics[width=\textwidth]{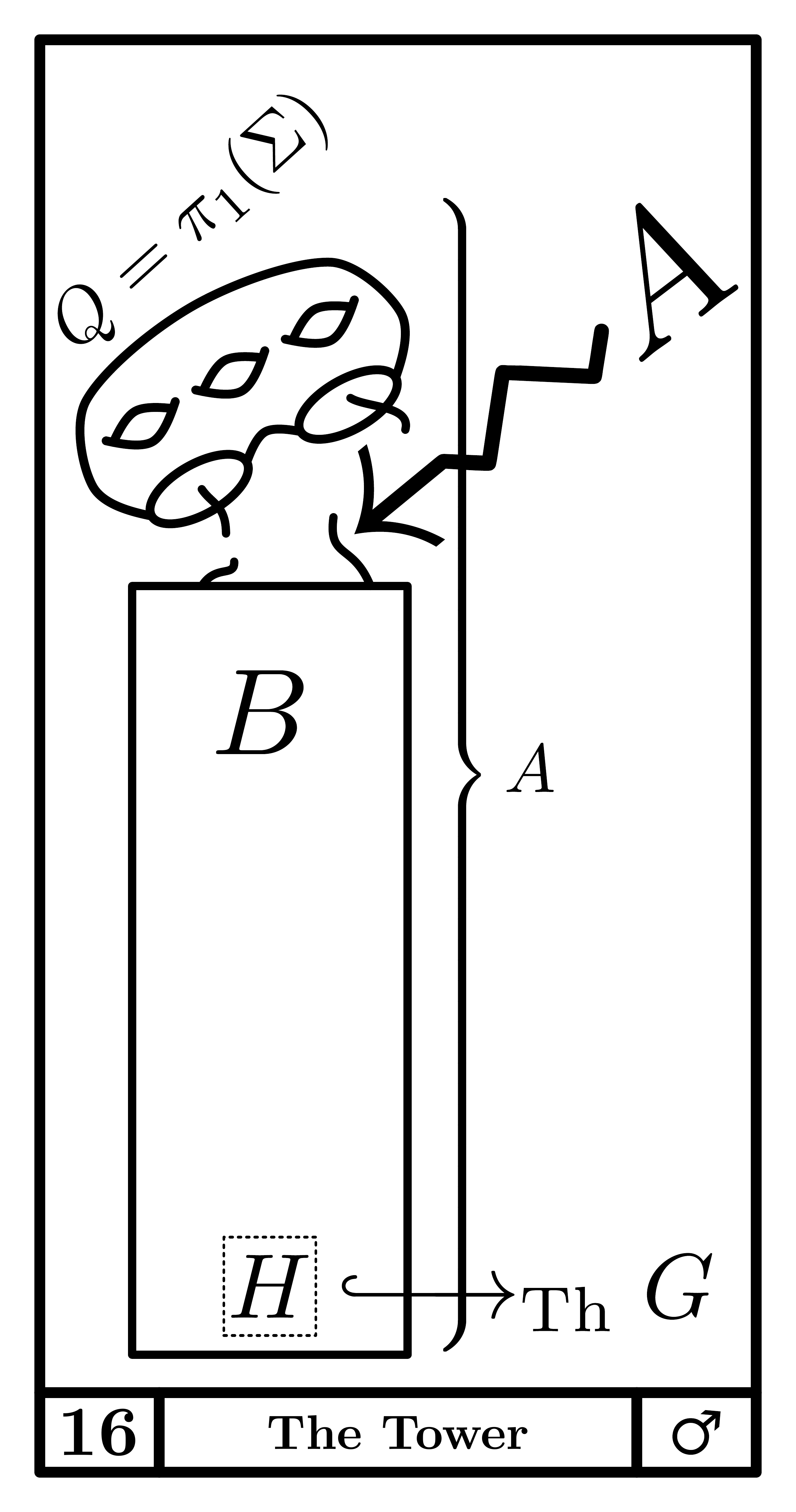}
\caption{A preretraction $A\sqto A$ ``breaks'' the ``top floor'' $Q$ off of $A$, leaving the base $B$ onto which $A$ retracts.}
\end{center}\end{subfigure}
\hfill
\begin{subfigure}[t]{0.46\textwidth}\begin{center}
\includegraphics[width=\textwidth]{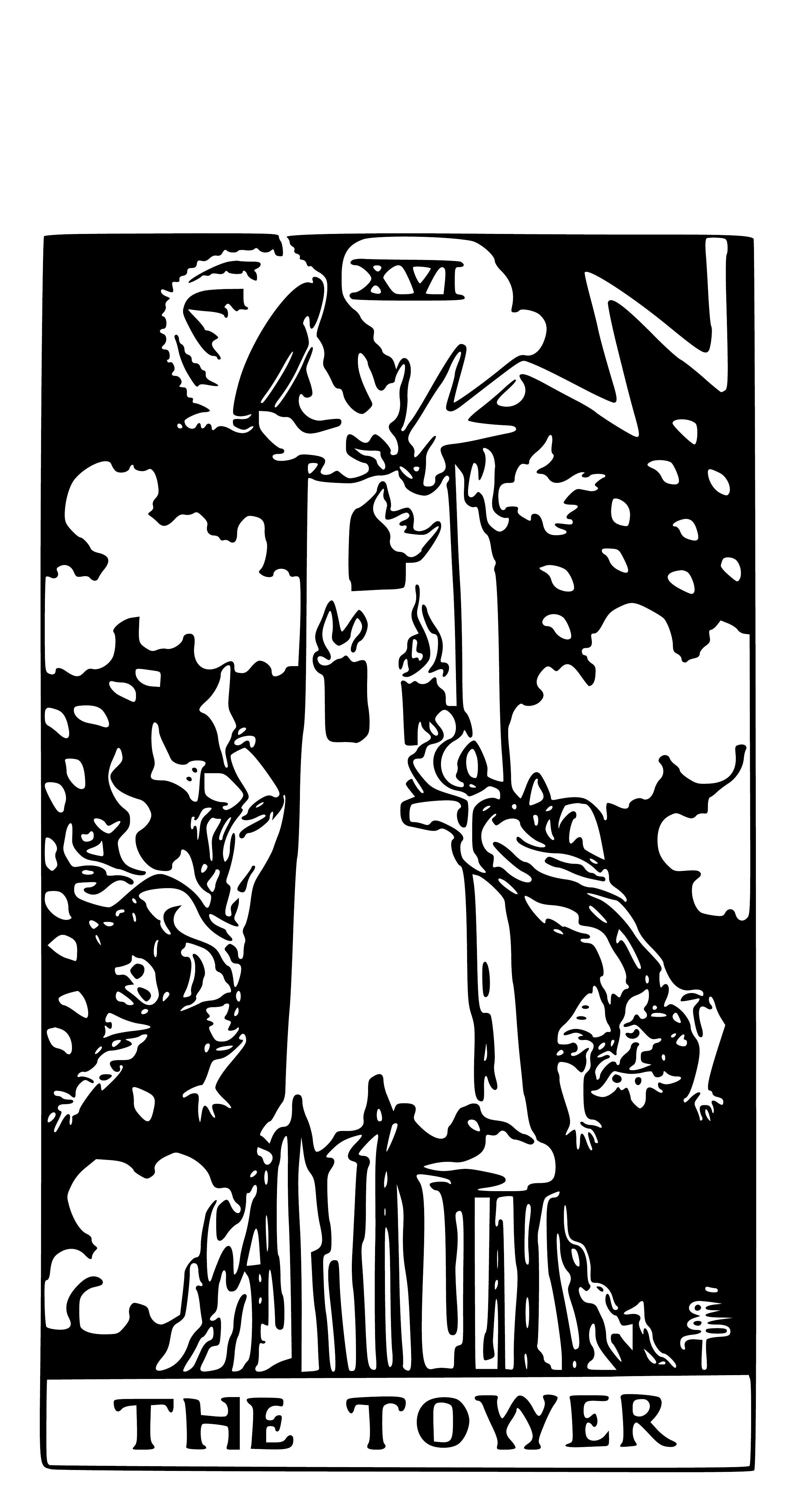}
\caption{\emph{The Tower} by Pamela Colman Smith, from the \emph{Rider-Waite} tarot deck (public domain).}
\end{center}\end{subfigure}

\caption{Towers, floors, and preretractions.}
\end{center}
\end{figure}

Building upon the work of Sela and others, Chlo\'e Perin proved the following:

\begin{thm}[{\cite[Theorem 1.2]{Perin11}}]
If $G$ is a torsion-free hyperbolic group and $H\elemb G$ is an elementary embedding, then $G$ is a tower over $H$.
\end{thm}

In this paper we prove a generalization of this result to \emph{toral relatively hyperbolic groups}, which are torsion-free groups hyperbolic relative to maximal abelian subgroups. For more on toral relatively hyperbolic groups, see \cite{Groves05,Groves09}.

\begin{dfn}
Let $G=\lang S\rang$ be a finitely generated group (with $S$ finite), let $\cp=\{P_1,\ldots,P_n\}$ be a collection of finitely generated subgroups of $G$, and let $X$ be the Cayley graph of $G$ with respect to $S$. We construct the \emph{coned-off Cayley graph} $\tilde{X}$ by joining a unique cone point for each distinct left coset of an element of $\cp$ to each vertex of that coset in $X$, i.e.,
\[V(\tilde{X})=G\cup\{c_{gP}:gP\text{ is a coset with }g\in G,P\in\cp\},\]
\[E(\tilde{X})=E(X)\cup\{(h,c_{gP}):h\in gP\text{ with }g\in G,P\in\cp\}.\]
\end{dfn}

\begin{dfn}
A finitely generated group $G$ is \emph{hyperbolic} if its Cayley graph is $\delta$-hyperbolic.

A finitely generated group $G$ is \emph{hyperbolic relative to} $\cp=\{P_1,\ldots,P_n\}$, a collection of finitely generated subgroups, if the coned-off Cayley graph $\tilde{X}$ is $\delta$-hyperbolic and, for each $e\in E(\tilde{X})$ and $n\in\bbn$, there are only finitely many embedded loops of length $n$ containing $e$.

A \emph{toral relatively hyperbolic group} is a torsion-free group which is hyperbolic relative to the conjugacy representatives of its maximal non-cyclic abelian subgroups. In particular, it follows that all elements of the set $\cp$ are non-cyclic.
\end{dfn}

\begin{restatable*}[]{mthm}{mainthms}
If $G$ is a toral relatively hyperbolic group and $H\elemb G$ is an elementary embedding, then $G$ is a tower over $H$.
\end{restatable*}

\subsection*{Acknowledgements}
I would like to thank my adviser, Daniel Groves, as well as Zlil Sela, Chlo\'e Perin, Vincent Guirardel, and Gilbert Levitt for all of their assistance in writing this paper and helping me to understand many of the concepts and techniques used here. 

I would like to especially thank Vincent and Gilbert for providing a draft of their forthcoming paper with Rizos Sklinos, \cite{GLS18}, which streamlines and clarifies the process of constructing floors and towers from preretractions. In \S 2--3, 4.2, 5.3 they only require their groups to be finitely generated, torsion-free, and possibly CSA, so their results can be cited directly and applied to toral relatively hyperbolic groups. As their paper is currently in preparation we have provided statements of relevant results for completeness, but much of what is cited from their draft is based on techniques used in \cite{Perin11,GLS17}.

I would also like to acknowledge the support of the Mellon-Mays Foundation, which funded the travel that enabled me to meet with Zlil, Chlo\'e, Gilbert, and Vincent.

Finally, I would like to thank my partner Stephanie Reyes for her constant support and encouragement. Without her suggestion that I just start typing up my ideas, this work would still be a mess of disorganized notes strewn across several notebooks and stacks of loose papers.

\subsection{Towers, splittings, and preretractions}\label{tap}

For more information on JSJ trees and splittings, see \cite{GL17}, and for more information on towers, preretractions, and their associated splittings, see \cite{GLS18}.

\begin{dfn}
A group $G$ is \emph{freely indecomposable relative to a subgroup $H$} if it does not admit a non-trivial free product decomposition $G=A*B$ with $H\leq A$.
\end{dfn}

\begin{dfn}
Let $G$ be a torsion-free group acting on a simplicial tree $T$ without edge inversion. This action is \emph{$k$-acylindrical} if the pointwise stabilizer of each arc of length $\geq k+1$ is trivial. Given a vertex $v\in V(T)$, this action is \emph{1-acylindrical at $v$} if, for any pair of distinct edges meeting at $v$, all conjugates of the corresponding edge stabilizers intersect trivially.
\end{dfn}

\begin{dfn}
Let $\Lambda$ be a splitting of a group $G$. A vertex $v\in V(\Lambda)$ is a \emph{surface-type vertex} (and similarly the vertex group $G_v$ is a \emph{surface-type vertex group}) if there exists a compact surface $\Sigma$ such that $G_v\cong\pi_1(\Sigma)$ and there is a bijective correspondence between the boundary components $\der\Sigma=C_1\amalg\cdots\amalg C_n$ of $\Sigma$ and the edges $e_1,\ldots,e_n$ incident to $v$ so that $G_{e_i}=\pi_1(C_i)\cong\bbz$ for all $i$.
\end{dfn}

\begin{dfn}
There are four classes of surfaces $\Sigma$ with $\chi(\Sigma)=-1$ which do not carry pseudo-Anosov diffeomorphisms:
\begin{itemize}
\item
pairs of pants;
\item
once-punctured Klein bottles;
\item
twice-punctured projective planes; and
\item
non-orientable closed surfaces of genus 3.
\end{itemize}
We refer to such surfaces and surface-type vertices in splittings which carry the fundamental groups of these surfaces as \emph{exceptional}. All other hyperbolic surfaces (and the associated surface-type vertices) are \emph{non-exceptional}. Because all of the surfaces we will consider have negative Euler characteristic, non-exceptional surfaces will admit pseudo-Anosov diffeomorphisms. Furthermore, non-exceptional surfaces all have $\chi(\Sigma)\leq -2$.
\end{dfn}

\begin{dfn}
A splitting $\Lambda$ of a finitely generated, torsion-free group $G$ is \emph{JSJ-like} if
\begin{itemize}
	\item
		edge groups are abelian;
	\item
		at most one vertex adjacent a given edge is a surface-type vertex, and at most one is an abelian vertex;
	\item
		the action of $G$ on the corresponding Bass-Serre tree $T$ is 2-acylindrical; and
	\item
		the surfaces of $\Lambda$ are punctured tori or are non-exceptional.
	\end{itemize}
Vertex groups which are neither abelian nor surface-type are called \emph{rigid}.

Note that this definition differs from those used in \cite{GLS18,Perin11} by generalizing to abelian splittings rather than just cyclic splittings.
\end{dfn}

\begin{dfn}\label{jsj}
Let $G$ be a non-abelian toral relatively hyperbolic group and let $H\leq G$ a subgroup. Let $\cp$ be a complete set of conjugacy representatives for the maximal non-cyclic abelian subgroups of $G$, let $\cb$ be the collection of subgroups of $G$ which are conjugate to either a subgroup in $\cp$ or a virtually cyclic subgroup, and let $\ch=\{H\}\cup\cp$. If $G$ is freely indecomposable relative to $H$, then by \cite[Corollary 9.19]{GL17} there exists the JSJ tree $T$ over $\cb$ relative to $\ch$ which is equal to its collapsed tree of cylinders, is invariant under automorphisms of $G$ which fix $\ch$, and is compatible with every $(\cb,\ch)$-tree. For more on the construction of the JSJ tree and the collapsed tree of cylinders, see \cite{GL17}.

From the construction of $T$ as a collapsed tree of cylinders we also have that $\Lambda:=T/G$ is bipartite with every edge carrying one abelian vertex and one non-abelian vertex, and that the action on $T$ is 1-acylindrical near vertices with non-abelian stabilizer. In particular the JSJ splitting $\Lambda$ is JSJ-like. We will refer to this splitting as the \emph{JSJ splitting of $G$ relative to $H$} (or simply the \emph{JSJ splitting of $G$} if $H=\{1\}$.
\end{dfn}

\begin{lem}[[{Compare \cite[Lemma 2.4]{GLS18}}]\label{lem24}
Let $A$ and $G$ be finitely generated torsion-free groups acting on trees $T_A$ and $T_G$, respectively, with abelian edge stabilizers. Assume that both trees are bipartite with verties of types 0 and 1, and that the actions are 1-acylindrical near vertices of type 1.

Let $f:A\to G$ be a homomorphism such that each type 0 vertex stabilizer subgroup of $A$ maps injectively into a type 0 vertex stabilizer of $G$, and that each type 1 vertex stabilizer of $A$ maps bijectively to a type 1 vertex stabilizer of $G$. If $f$ is not injective, then there exist two non-conjugate type 1 vertex stabilizer subgroups of $A$ with the same image under $f$.
\end{lem}

\begin{proof}
The 1-acylindricity condition ensures that edges which meet at a type 1 vertex are malnormal, so because edge stabilizers are abelian we have that type 1 vertex stabilizers are non-abelian and fix a unique type 1 vertex. Type 0 vertex stabilizers similarly fix unique type 0 vertices. Given a vertex $v\in V(T_A)$, let $\ph(v)\in V(T_G)$ be the unique vertex which is of the same type as $v$ and is fixed by $f(A_v)\leq G_{\ph(v)}$. By 1-acylindricity $\ph$ preserves adjacency, so we can extend to a map $\ph:T_A\to T_G$ which maps edges to edges. Because $f$ is non-injective but is injective on vertex stabilizers, there exist distinct edges $e=(u,v),e'=(u',v)\in E(T_A)$ such that $\ph(u)=\ph(u')$ and $\ph(e)=\ph(e')$. In particular, we must have that $v$ is type 0 and both $u$ and $u'$ are type 1 by 1-acylindricity.

If $u$ and $u'$ lie in different orbits then the result follows immediately, so suppose that there exists $a\in A$ such that $u'=a\cdot u$. Then $f(a)\cdot\ph(u)=\ph(a\cdot u)=\ph(u')=\ph(u)$, so $f(a)\in G_{\ph(u)}=f(A_u)$ and hence there exists $b\in A_u$ such that $f(a)=f(b)$. Let $c=ab^{-1}$. Then $c\cdot u=a\cdot(b^{-1}\cdot u)=a\cdot u=u'$ and $c\in\ker(f)$. In particular $c\cdot v\neq v$ because $f$ is injective on vertex stabilizers. Let $d\in A_e$ and let $d'\in A_{e'}\leq A_{u'}$ be nontrivial. Then $cdc^{-1}\in A_{c\cdot e}\leq A_{u'}$, and by 1-acylindricity near $u'$ we have that $D:=\lang cdc^{-1},d'\rang$ is non-abelian. But then we must have that $f(D)=\lang f(cdc^{-1}),f(d')\rang=\lang f(d),f(d')\rang$ is abelian because $f(D)$ is a subgroup of the abelian group $G_{\ph(e)}$, which contradicts the injectivity of $f$ on $A_{u'}$.
\end{proof}

\begin{dfn}
A cyclic splitting $\Lambda$ of a finitely generated torsion-free group $A$ is \emph{centered} if it has vertices $v,v_1,\ldots,v_n$ with $n\geq 1$ such that $v$ is surface-type and every edge joins $v$ to some $v_i$. Let $Q=\pi_1(\Sigma)=A_v$, let $B_i=A_{v_i}$, and let $T$ be the Bass-Serre tree with respect to this splitting. We further require that the splitting $\Lambda$ is \emph{minimal} in the sense that, if $v_i$ is incident to only one edge, then that edge group must be a proper subgroup of $B_i$.

We will refer to $v$ as the \emph{central vertex} and each $v_i$ as \emph{exterior vertices}. The \emph{base} of $\Lambda$ is the (abstract) free product $B_\Lambda:=B_1*\cdots *B_n$. We say that a centered splitting is \emph{simple} if $n=1$, and that it is \emph{non-exceptional} if $\Sigma$ is non-exceptional.
\end{dfn}

\begin{figure}[t]\label{centered}
\begin{center}
\includegraphics[width=.5\textwidth]{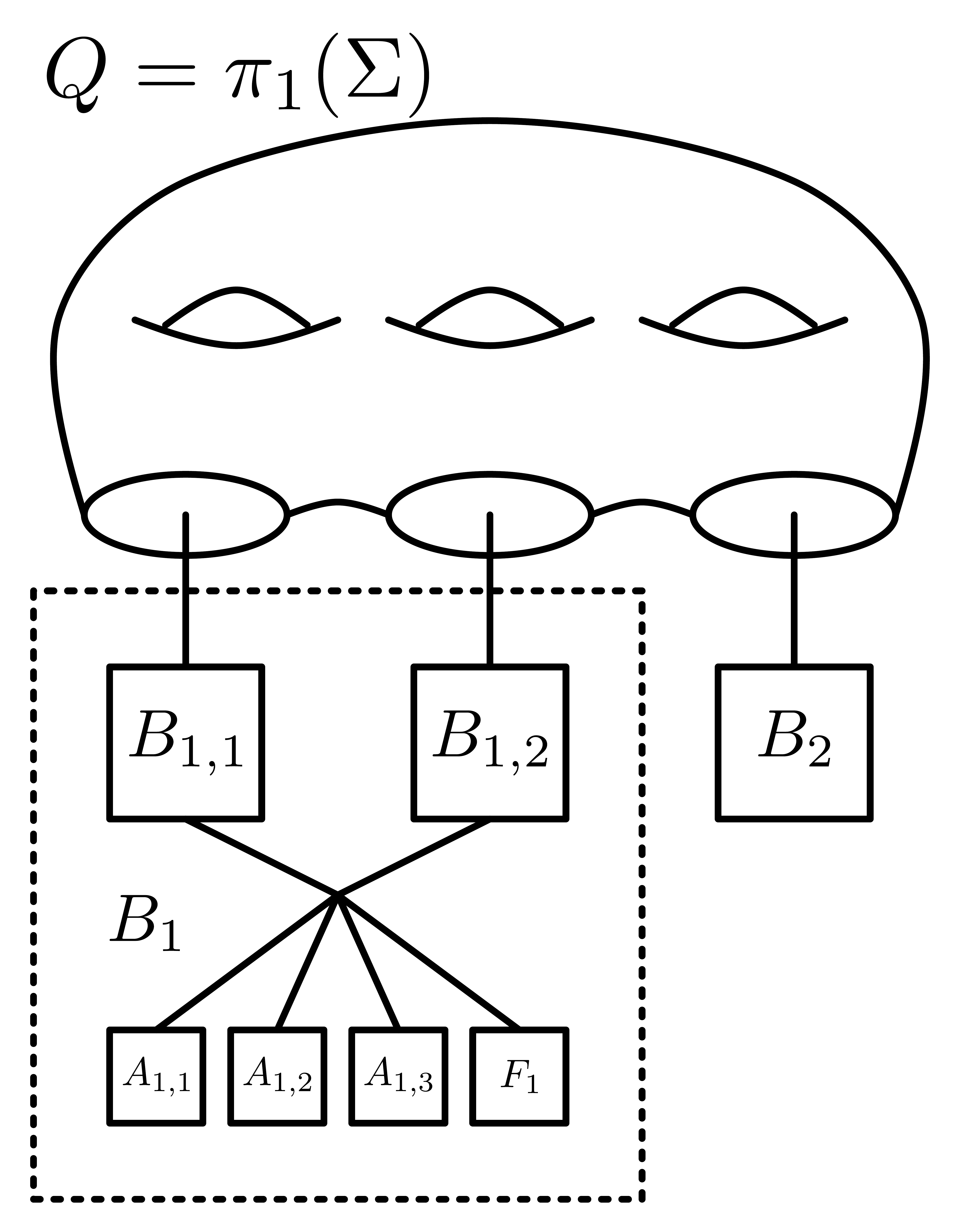}
\caption[A centered splitting.]{A centered splitting with the exterior vertex $B_1$ ``blown up'' to show its Grushko decomposition relative to the boundary subgroups.}
\end{center}
\end{figure}

\begin{dfn}
Let $Q=\pi_1(\Sigma)$ be the fundamental group of a compact surface and suppose that $Q$ is a subgroup of $G$. A \emph{boundary-preserving map (with values in $G$)} is a morphism $p:Q\to G$ (denoted $p:Q\sqto G$) which restricts to conjugation by an element of $G$ on each boundary subgroup of $\pi_1(\Sigma)$. Such a map is \emph{non-degenerate} if, additionally, $p(Q)$ is non-abelian, $p$ is not an isomorphism onto a conjugate of $Q$, and $\Sigma$ is non-exceptional. If $Q=A_v$ is a surface-type vertex in some splitting $\Lambda$ of $A\leq G$, we say that $p$ is \emph{with respect to $\Lambda$}.
\end{dfn}

\begin{dfn}
Let $Q=\pi_1(\Sigma)$ be the fundamental group of a compact surface and let $p:Q\to G$ be a group homomorphism. A 2-sided simple closed curve $\gamma$ is a \emph{pinched curve} if it is not nullhomotopic and if $\pi_1(\gamma)\leq\ker(p)$. A \emph{family of pinched curves} is a family of disjoint, pairwise non-parallel pinched curves. The map $p$ is a \emph{pinching map} if there is a pinched curve.
\end{dfn}

\begin{lem}[{\cite[Lemma 3.9]{GLS18}}]\label{lem39}
Let $\Lambda$ be a centered splitting of $G$ with central vertex group $Q=\pi_1(\Sigma)$. Let $S$ be another compact surface and let $p:\pi_1(S)\to G$ such that the image of each boundary subgroup of $\pi_1(S)$ is contained in a conjugate of an exterior vertex group of $\Lambda$.

Let $\cc$ be a maximal family of pinched curves on $S$ and let $S'$ be a component of $S$ obtained by cutting $S$ along $\cc$. Then either
\begin{enumerate}[(i)]
\item
$p(\pi_1(S'))$ is contained in a conjugate of an exterior vertex group of $\Lambda$; or
\item
there is an incompressible subsurface $Z\subseteq S'$ such that $p(\pi_1(Z))$ is finite index in $Q$ and $p$ maps boundary subgroups of $\pi_1(Z)$ into boundary subgroups of $Q$.
\end{enumerate}
\end{lem}

\begin{prop}[{\cite[Proposition 3.17]{GLS18}}]\label{extendedprop}
Let $\Lambda$ be a non-exceptional centered splitting of a finitely generated torsion-free group $A$ with central vertex subgroup $Q$ and exterior vertex subgroups $B_1,\ldots,B_n$.
\begin{enumerate}[(1)]
\item
If $n>1$ or if $n=1$ and $B_1\not\cong\bbz$, then there exists a non-degenerate boundary-preserving map $p:Q\sqto A$ if and only if there exist conjugates $\tilde{B_i}$ of each $B_i$ such that
\[B:=\lang\tilde{B_1},\ldots,\tilde{B_n}\rang\cong\tilde{B_1}*\cdots*\tilde{B_1}\cong B_\Lambda\]
and there is exists a retraction $r:A\surj B$ with $r(Q)$ non-abelian.

\item
If $n=1$ and $B:=B_1\cong\bbz$, then there exists a non-degenerate boundary preserving map $p:Q\sqto A$ if and only if there exists a retraction $r:A*\bbz\surj B*\bbz$ with $r(Q)$ non-abelian.
\end{enumerate}
\end{prop}

\begin{dfn}\label{retdef}
Let $A$ be a finitely generated torsion-free subgroup of $G$. A splitting $\Lambda$ of $A$ is \emph{retractable in $G$} if there exists a non-exceptional surface-type vertex $Q=\pi_1(\Sigma)=A_v$ and a non-degenerate boundary-preserving map $p:Q\sqto G$. If $A=G$, we simply say that $\Lambda$ is retractable.

In particular, a centered splitting $\Lambda$ of $A$ is retractable if and only if it satisfies the equivalent conditions of Proposition \ref{extendedprop}, and if $\Lambda$ is a retractable splitting of $A$, then so is the centered splitting obtained from $\Lambda$ by collapsing edges not carrying $v$.
\end{dfn}

\begin{dfn}
A group $G$ is a \emph{floor} over a subgroup $H$ if either $G=H*\bbz$ or if $G$ has a retractable centered splitting with base isomorphic to $H$.

A group $G$ is a \emph{tower} over a subgroup $H$ if there exists a chain of subgroups $G=G^0>G^1>\cdots>G^m=H$ with each $G^i$ a floor over $G^{i+1}$. This tower is \emph{trivial} if $m=0$. 
\end{dfn}

\begin{exa}[{\cite[Example 3.26]{GLS18}}]\label{towerover1}
All non-exceptional surface groups are towers over $\{1\}$.
\end{exa}

\begin{lem}[{\cite[Remarks 3.21--22]{GLS18}}]\label{gtower}
\begin{enumerate}[(1)]
\item
If $G$ is a floor over $H$, then $G*G'$ is a floor over $H*G'$.
\item
If $G$ is a tower over $H$, then $G*G'$ is a tower over $H*G'$.
\item
If $G>H>K$ are floors in a tower such that $G=H*\bbz$ and $H$ has a retractable centered splitting with base isomorphic to $K$, then $G$ has a retractable centered splitting with base isomorphic to $H'=K*\bbz$, making $G>H'>K$ a tower, i.e., we may assume that floors at the tops of towers are surface-type.
\end{enumerate}
\end{lem}

\begin{proof}
It is clear that (2) follows from (1), so assume $G$ is a floor over $H$. If $G=H*\bbz$, then $G*G'=(H*\bbz)*G'=(H*G')*\bbz$ is a floor over $H*G'$. If $G$ has a retractable centered splitting with base isomorphic to $H$, then we can make $G*G'$ a floor over $H*G'$ by replacing the base of this splitting with $H*G'$. Then (3) follows similarly.
\end{proof}

\begin{lem}[{\cite[Proposition 3.31]{GLS18}}]\label{prop331}
If $A$ is a free factor of $G$ and a splitting of $A$ is retractable in $G$, then it is retractable in $A$.
\end{lem}

\begin{lem}[{\cite[Lemma 4.12]{GLS18}}]\label{lem412}
Let $\Lambda_G$ be a retractable centered splitting of $G$ with central vertex group $Q_G=\pi_1(\Sigma_G)$. Let $A$ be an exterior vertex group of $\Lambda_G$, let $Q_A=\pi_1(\Sigma_A)$ be a proper surface subgroup of $A$, and let $p:Q_A\sqto G$ be a non-degenerate boundary-preserving map. Then there exists a non-degenerate boundary-preserving map $p':Q_A\sqto A$ if either
\begin{enumerate}[(i)]
\item
no conjugate of $p(Q_A)$ contains a finite index subgroup of $Q_G$; or
\item
$p$ is pinching.
\end{enumerate}
\end{lem}

\begin{dfn}
Let $G$ be a group and let $g\in G$. We will denote the \emph{inner automorphism} $x\mapsto gxg^{-1}$ by $\iota_g\in\Aut(G)$.
\end{dfn}

\begin{dfn}
Let $\Lambda$ be an abelian splitting of a group $A$. We say that two morphisms $f,f':A\to G$ are \emph{$\Lambda$-related} if
\begin{itemize}
\item
for each edge $e$ of $\Lambda$, there exists $g_e\in G$ such that $f'|_{A_e}=\iota_{g_e}\circ f|_{A_e}$.
\item
for each vertex $v$ of $\Lambda$ which is either an exceptional surface-type vertex or a non-surface-type vertex, there exists $g_v\in G$ such that $f'|_{A_v}=\iota_{g_v}\circ f|_{A_v}$; and
\item
for each non-exceptional surface-type vertex $v$ of $\Lambda$, $f(A_v)$ is non-abelian if and only if $f'(A_v)$ is non-abelian.
\end{itemize}
\end{dfn}

\begin{dfn}
Let $A$ be a subgroup of $G$ and let $\Lambda$ be an abelian splitting of $A$. A morphism $p:A\to G$ is a \emph{preretraction (with respect to $\Lambda$, with values in $G$)} if it is $\Lambda$-related to the inclusion $A\emb G$. As with boundary-preserving maps, we will denote preretractions $p:A\sqto G$.
\end{dfn}

\begin{prop}[{Compare \cite[Proposition 5.17]{GLS18}}]\label{floor}
Let $G$ be a tower over $G'$ and let $A$ be a free factor of $G'$. Suppose that $\Lambda$ is an abelian splitting of $A$ which is bipartite with every edge carrying one abelian vertex and one non-abelian vertex, and that the Bass-Serre tree of $\Lambda$ is 1-acylindrical near vertices with non-abelian stabilizer. If there exists a non-injective preretraction $p:A\sqto G$ with respect to $\Lambda$, then $\Lambda$ is retractable in $A$.
\end{prop}

\begin{proof}
First suppose that $G'=G$, so that $A$ is a free factor of $G$. If some surface-type vertex subgroup $Q\leq A$ is not mapped isomorphically to a conjugate, then $p|_Q:Q\sqto G$ is a non-degenerate boundary-preserving map, so $\Lambda$ is retractable in $G$ and hence in $A$ by Lemma \ref{prop331}. Otherwise if every surface-type vertex group is mapped isomorphically to a conjugate, let $r:G\surj A$ be the retraction of $G$ onto its free factor $A$. But then $r\circ p:A\to A$ is injective by Lemma \ref{lem24}, which contradicts the assumption that $p$ is non-injective.

Now suppose that $G=G^0>G^1>\cdots>G^m=G'$ with each $G^i$ a floor over $G^{i+1}$. By Lemma \ref{gtower} we may assume that each floor is of surface type, so let $\Gamma_i$ be the retractable centered splitting of $G^i$ with central vertex group $Q_i$ and base isomorphic to $G^{i+1}$. Fix retractions $r_i:G^i\surj G^{i+1}$ and define $p_i:=r_{i-1}\circ r_{i-2}\circ\cdots\circ r_0\circ p:A\to G^i$
with $p_0=p$. Because $A$ is freely indecomposable it is contained in a conjugate of an exterior vertex subgroup of each $\Gamma_i$ and $r_i|_A$ agrees with conjugation. If $p_m$ is a preretraction then the result follows from the previous case. Otherwise the splitting $\Lambda$ has a non-exceptional surface-type vertex group $Q$ such that $p_m(Q)$ is abelian. Because $p(Q)$ is non-abelian there exists some maximal index $i$ such that $p_i(Q)$ is non-abelian.

Let $p':=p_i|_Q:Q\to G_i$. Then $p'$ is a boundary-preserving map and $p'(Q)$ is not conjugate to $Q$ because $r_i|_A$ is injective and $p_{i+1}(Q)=r_i\circ p_i(Q)$ is abelian by assumption. Thus $p'$ is non-degenerate. Furthermore, $p'(Q)$ cannot contain a finite index subgroup of $Q_i$ because $r_i(Q_i)$ is non-abelian, nor can $p'(Q)$ be contained in an exterior vertex group of $\Gamma_i$. Thus $p'$ is pinching by Lemma \ref{lem39}. By applying Lemma \ref{lem412} $m-i$ times we obtain a pinching non-degenerate boundary-preserving map $Q\sqto G^m=G'$. Then $\Lambda$ is retractable in $H$ and hence in $A$ by Lemma \ref{prop331}.
\end{proof}

\subsection{Proof of the main theorem}

\begin{lem}
If $G$ is a toral relatively hyperbolic group and $H$ is a retract of $G$, then $H$ is also a toral relatively hyperbolic group.
\end{lem}

\begin{proof}
Let $G=\lang S_G\rang$ with $S_G$ finite and let $r:G\surj H$ be the retraction. Then $H=\lang r(S_G)\rang$ is finitely generated as well. Without loss of generality, we may choose $S_G$ so that $r(s)=s$ or $r(s)=1$ for all $s\in S_G$. Otherwise if $r(s)=t$ with $t\neq s$ and $t\neq 1$, we can define a new generating set $(S_G\setminus\{s\})\cup\{st^{-1},t\}$ with $r(t)=t$ and $r(st^{-1})=1$. In particular, we can choose $S_H\subseteq S_G$ so that $H=\lang S_H\rang$ and
\[r(s)=\pw{s & s\in S_H\\ 1 & s\notin S_H}.\]

Let $(X_G,d_G)$ and $(X_H,d_H)$ be the Cayley graphs of $G$ and $H$ with respect to $S_G$ and $S_H$, respectively. Suppose that $a,b\in H$ with $d_G(a,b)=n$. Then there exist $s_1,\ldots,s_n\in S_G$ such that $a^{-1}b=s_1\cdots s_n$. It is clear that $s_i\in S_H$ for all $i$, as otherwise we could find a shorter word by applying the retraction $r$. Thus $d_H(a,b)=n=d_G(a,b)$, so $H$ is quasi-isometrically embedded in $G$. Then $H$ is relatively quasiconvex by \cite[Theorem 1.5]{Hruska10} and hence a toral relatively hyperbolic group by \cite[Theorem 1.2]{Hruska10}.
\end{proof}

The proof of the main theorem relies on the following, which are the main technical results of this paper and are proved in \S\ref{sec6}.

\begin{restatable*}[Compare {\cite[Proposition 5.13]{Perin11}}]{mprop}{mainpropa}\label{global}
Let $G$ be a toral relatively hyperbolic group and let $H\elemb G$ be an elementarily embedded subgroup. If $A$ is a retract of $G$ which properly contains $H$, is toral relatively hyperbolic, and is freely indecomposable relative to $H$, then there exists a non-injective preretraction $p:A\sqto G$ with respect to the JSJ splitting of $A$ relative to $H$.
\end{restatable*}

\begin{restatable*}[Compare {\cite[Proposition 5.14]{Perin11}}]{mprop}{mainpropb}\label{global1}
Let $G$ be a toral relatively hyperbolic group and let $H\elemb G$ be an elementarily embedded subgroup. If $H$ is a toral relatively hyperbolic retract of $G$ and $C$ is a freely indecomposable torsion-free hyperbolic subgroup of $G$ such that no non-trivial elements of $C$ are conjugate into $H$ by an element of $G$ and $C$ is neither cyclic nor a non-exceptional surface group, then there exists a non-injective preretraction $p:C\sqto G$ with respect to the JSJ splitting of $C$.
\end{restatable*}

\mainthms

\begin{proof}
If $G\cong\bbz^r$ is abelian then $H=G$ because $\bbz^r$ has no proper elementarily embedded subgroups (see Lemma \ref{finite-index}). Then trivially $G$ is a tower over $H$, so suppose that $G$ is a non-abelian toral relatively hyperbolic group and $H$ is a proper subgroup of $G$.

Let $G^0=G$, and assuming that $G$ is a tower over a toral relatively hyperbolic group $G^m$ containing $H$ for $m\geq 0$, let
\[G^m=A^m*C_1^m*\cdots*C_{k_m}^m\]
be the Grushko decomposition of $G^m$ relative to $H$, where $A^m$ is the factor containing $H$. Suppose that $A^m\neq H$, and let $\Lambda$ be the JSJ splitting of $A^m$ relative to $H$. Then by Proposition \ref{global} there exists a non-injective preretraction $A^m\sqto G$ with respect to $\Lambda$, so $\Lambda$ is retractable in $A^m$ by Proposition \ref{floor}. We can then collapse edges of $\Lambda$ as in Definition \ref{retdef} to obtain a retractable centered splitting of $A^m$, making $A^m$ a floor over the base $B^m$ of this splitting. Then $B^m$ is a toral relatively hyperbolic group because it is the image of a series of retractions
\[G=G^0\surj G^1\surj\cdots\surj G^m\surj A^m\surj B^m\]
by the definition of a tower, so we define $G^{m+1}=B^m*C_1^m*\cdots*C_{k_m}^m$ making $G$ a tower over $G^{m+1}$.

Thus we obtain a sequence of $G$-limit groups
\[G=G^0>G^1>G^2>\cdots\]
which terminates at some $G^N$ by \cite[Theorem 5.2]{Groves05}. Because this sequence terminates we must have $G^N=H*C_1*\cdots*C_k$. Each $C_i$ is a toral relatively hyperbolic group because each is a retract of $G$, and in particular they are all torsion-free hyperbolic groups because the conjugates of $H$ and $C_i$ intersect trivially in $G^N$ and $H$ is torally complete in both $G$ and $G^N$. If each $C_i$ is either cyclic or a non-exceptional surface group, then each $C_i$ would be a tower over $\{1\}$ by Example \ref{towerover1}.

Suppose that some factor $C_i$ is neither cyclic nor a non-exceptional surface group. Conjugates of $H$ and $C_i$ intersect trivially in $G^N$ and hence in $G$ because $G^N$ is a retract of $G$. By Proposition \ref{global1} we obtain a non-injective preretraction $p:C_i\sqto G$ with respect to the JSJ splitting $\Lambda$ of $C_i$. By Proposition \ref{floor} we obtain a retractable centered splitting of $C_i$, making $C_i$ a floor over the base $D_i$ of this splitting. Because $D_i$ is also a retract of $G$ we may apply this process again to the Grushko decomposition of $D_i$ if any factors are neither cyclic nor a non-exceptional surface group. This process again must terminate, so we find that $C_i$ is a tower over $\{1\}$.

Then $G^N=H*C_1*\cdots*C_k$ is a tower over $H*1*\cdots*1\cong H$ by Lemma \ref{gtower}, so $G$ is a tower over $H$.

\end{proof}
\begin{cor}\label{maincor}
If $G$ is a toral relatively hyperbolic group and $H\elemb G$ is an elementary embedding, then $H$ is a toral relatively hyperbolic group.
\end{cor}

\section{First-order logic}

For $n\geq 0$, we will use the notations $\ux$ or $(\ux)$ to denote finite \emph{ordered} tuples $(x_1,\ldots,x_n)$. A sequence of elements $x_n\in X$ for $n\in\bbn$ will be denoted $(x_n)_n$.

In first-order formulas, we will use the notations $\exists\ux$ and $\forall\ux$ to denote $\exists x_1\exists x_2\ldots\exists x_n$ and $\forall x_1\forall x_2\ldots\forall x_n$, respectively.

Given first-order formulas $\ph_1(\ux),\ldots,\ph_m(\ux)$ we will use $\uph(\ux)$ to denote $\ph_1(\ux)\wedge\cdots\wedge\ph_m(\ux)$, and similarly given words $\sigma_1(\ux),\ldots,\sigma_m(\ux)\in\bbf(\ux):=\lang\ux\rang$ we will use $\usig(\ux)=1$ to denote $(\sigma_1(\ux)=1)\wedge\cdots\wedge(\sigma_m(\ux)=1)$.

More generally, if $\Sigma=\Sigma(\ux)\subseteq\bbf=\lang\ux\rang$ is a (possibly infinite) set of words in the variables $\ux$, we denote
\[\Sigma(\ux)=1\iff\sigma(\ux)=1\text{ for all }\sigma\in\Sigma\]
for a system of equations, and similarly
\[\Sigma(\ux)\neq 1\iff\sigma(\ux)\neq 1\text{ for all }\sigma\in\Sigma\]
for a system of inequations.

\begin{lem}\label{presentation}
If $G=\lang\us\mid\usig(\us)=1\rang$ is a finitely presented group, then $G\models\exists\ux\ (\usig(\ux)=1)$ as witnessed by $\ux=\us$.
\end{lem}

\begin{dfn}
Let $G$ be a group. The \emph{elementary theory} of $G$ is the set of first-order sentences in the language of groups modeled by $G$,
\[\Th(G):=\{\ph:G\models\ph\}.\]
Two groups $G$ and $H$ are said to be \emph{elementarily equivalent} if $\Th(G)=\Th(H)$.

Let $H$ be a subgroup of a group $G$. The inclusion of $H$ into $G$ is an \emph{elementary embedding}, denoted $H\elemb G$, if for any first-order formula $\ph(\ux)$ and any $\uh\in H$,
\[H\models\ph(\uh)\iff G\models\ph(\uh).\]
In particular, $\Th(H)=\Th(G)$.
\end{dfn}

\begin{dfn}
A group $G$ is \emph{commutative-transitive} if for all $x,y,z\in G\setminus\{1\}$ we have
\[[x,y]=[y,z]=1\implies[x,z]=1.\]
A subgroup $H$ of a group $G$ is \emph{malnormal} if $gHg^{-1}\cap H=\{1\}$ for all $g\in G\setminus H$. A group $G$ is \emph{CSA} if every maximal abelian subgroup of $G$ is malnormal.
\end{dfn}

\begin{lem}
CSA groups are commutative-transitive.
\end{lem}

\begin{proof}
Let $x,y,z\in G\setminus\{1\}$ and suppose that $[x,y]=[y,z]=1$. Let $M$ be the maximal abelian subgroup  of $G$ containing both $x$ and $y$. If $z\notin M$ then $y=zyz^{-1}\in zMz^{-1}\cap M=\{1\}$, which contradicts the assumption that $y\neq 1$.
\end{proof}

\begin{rmk}
Because CSA groups are commutative-transitive, the centralizer $Z_G(g)$ of an element $g\in G\setminus\{1\}$ is the maximal abelian subgroup of $G$ containing $g$.
\end{rmk}

\begin{lem}
A group $G$ is commutative-transitive if and only if
\[G\models\forall x,y,z\ ((y\neq 1)\wedge([x,y]=1)\wedge([y,z]=1))\implies([x,z]=1).\]
\end{lem}

\begin{lem}\label{finite-index}
Let $G$ be a commutative-transitive group, let $H\elemb G$ be an elementarily embedded subgroup, and let $h\in H\setminus\{1\}$.
\begin{enumerate}[(1)]
\item
$Z_H(h)\leq Z_G(h)$.
\item
If $G$ is torsion-free and $[Z_G(h):Z_H(h)]<\infty$, then $Z_H(h)=Z_G(h)$.
\end{enumerate}
\end{lem}

\begin{proof}
(1) is immediate from the definition of centralizers, so to prove (2) suppose that $Z_H(h)\lneq Z_G(h)$. Let $z\in Z_G(h)\setminus Z_H(h)$ and let $n>1$ be the minimal integer such that $z^n\in Z_H(h)$. Then
\[G\models\exists x\ (x^n=z^n)\wedge([x,h]=1)\wedge([x,z]=1),\]
so
\[H\models\exists x\ (x^n=z^n)\wedge([x,h]=1)\wedge([x,z]=1)\]
and hence there exists $y\in Z_H(h)$ such that $y^n=z^n$ and $[y,z]=1$. Then $1=y^nz^{-n}=(yz^{-1})^n$, so because $G$ is torsion-free we must have $yz^{-1}=1$ and $z=y\in Z_H(h)$, which is a contradiction.
\end{proof}

\begin{lem}\label{ranks}
Suppose that $G$ is a torsion-free commutative-transitive group with all abelian subgroups finitely generated. For $r\geq 1$, there exist first-order formulas $\cz_r(x)$ such that for any $g\in G\setminus\{1\}$, the centralizer $Z_G(g)$ is a free abelian group of rank $r$ if and only if $G\models\cz_r(g)$.
\end{lem}

\begin{proof}
Note that the centralizer $Z_G(g)$ is a free abelian group because $G$ is torsion-free, and that, for all $r,s\geq 1$,
\[\bbz^r\models\exists t_1,\ldots,t_s\forall x\exists y\ \bigvee_{(\epsilon_i)\in\{0,1\}^s} x=y^2t_1^{\epsilon_1}\cdots t_s^{\epsilon_s}\]
if and only if $r\leq s$. This follows because the formula states that $H_1(\bbz^r;\bbz/2\bbz)$ contains at most $2^s$ elements, and hence we may use this formula to distinguish maximal free abelian subgroups of different ranks.

Then we may define the formula
\begin{align*}
	\tilde{\cz}_r(g)\ :\
		&\exists t_1,\ldots,t_r\forall x\ \left(\bigwedge_{1\leq i\leq r}[g,t_i]=1\right)\\
		&\wedge\left(([g,x]=1)\implies \left(\exists y\ ([g,y]=1)\wedge\left(\bigvee_{(\epsilon_i)\in\{0,1\}^r} x=y^2t_1^{\epsilon_1}\cdots t_r^{\epsilon_r}\right)\right)\right)
	\end{align*}
so that $G\models\tilde{\cz}_r(g)$ if and only if $\rank(Z_G(g))\leq r$. Define the formulas
\begin{align*}
\cz_1(g)\ &:\ \tilde{\cz}_1(g)\wedge(g\neq 1)\\
\cz_r(g)\ &:\ \neg\tilde{\cz}_{r-1}(g)\wedge\tilde{\cz}_r(g)\wedge(g\neq 1) & (r>1)
\end{align*}
Then $G\models\cz_r(g)$ if and only if $\rank(Z_G(g))=r$.

\end{proof}\begin{lem}\label{max-abelian}
Suppose that $G$ is a torsion-free commutative-transitive group with all abelian subgroups finitely generated and finitely many conjugacy classes of maximal non-cyclic abelian subgroups. There exists a first-order formula $\ca_G(\ux)$ such that $G\models\ca_G(\ua)$ if and only if $\{Z_G(\ua)\}$ forms a complete set of conjugacy representatives for the maximal non-cyclic abelian subgroups of $G$ with $1<\rank(Z_G(a_i))\leq\rank(Z_G(a_{i+1}))$ for all $i$. In particular, $G\models\exists\ux\ \ca_G(\ux)$.
\end{lem}

\begin{proof}
Let $Z_1,\ldots,Z_n\leq G$ be a maximal set of non-conjugate, maximal non-cyclic abelian subgroups of $G$. Each $Z_i$ is free abelian, so let $r_i=\rank(Z_i)>1$ and order these subgroups so that $r_i\leq r_{i+1}$ for all $i$. Define the formula
\begin{align*}
	\ca_G(\ux)\ :\
		&\left(\bigwedge_{1\leq i\leq n} \cz_{r_i}(x_i)\right)\wedge \left(\forall w\ \bigwedge_{1\leq i<j\leq n} [x_i,wx_jw^{-1}]\neq 1]\right)\\
		&\wedge\left(\forall y\ (y=1)\vee(\cz_1(y))\vee\left(\exists z\ \bigvee_{1\leq i\leq n} [y,zx_iz^{-1}]=1\right)\right)
	\end{align*}
where $\cz_r(x)$ is as in Lemma \ref{ranks}. Then $G\models\exists\ux\ \ca_G(\ux)$, where any solution $\ux=\ua$ with $a_i\in Z_i\setminus\{1\}$ satisfies this sentence. Furthermore, $G\models\ca_G(\ua)$ if and only if $\{Z_G(\ua)\}$ forms a complete set of conjugacy representatives for the maximal non-cyclic abelian subgroups of $G$ with $1<\rank(Z_G(a_i))\leq\rank(Z_G(a_{i+1}))$ for all $i$.
\end{proof}

\begin{lem}[{\cite[Lemma 2.2]{Groves05}}]
If $G$ is a toral relatively hyperbolic group, then all non-cyclic abelian subgroups of $G$ are finitely generated.
\end{lem}

\begin{lem}[{\cite[Lemma 2.5]{Groves09}}]
Toral relatively hyperbolic groups are CSA and hence commutative-transative.
\end{lem}

\begin{dfn}
A subgroup $H$ of a toral relatively hyperbolic group $G$ is \emph{torally complete} in $G$ if for every maximal non-cyclic abelian subgroup $Z\leq G$ there exists $g\in G$ such that $gZg^{-1}\leq H$.
\end{dfn}

\begin{lem}
Let $H$ be a subgroup of a toral relatively hyperbolic group $G$. If $H\elemb G$, then $H$ is torally complete in $G$.
\end{lem}

\begin{proof}
Suppose that $H\neq G$. The subgroup $H$ is also torsion-free, and further it is commutative-transitive because this property can be determined with first-order logic. Let $\ca_G(\ux)$ as in Lemma \ref{max-abelian}. Then $H\models\exists\ux\ \ca_G(\ux)$, so let $\ua\in H$ such that $H\models\ca_G(\ua)$. Then $G\models\ca_G(\ua)$ by the elementary embedding, so $\{Z_G(\ua)\}$ is a complete set of conjugacy class representatives for the maximal non-cyclic abelian subgroups of $G$. We also have that $\rank(Z_H(a_i))=\rank(Z_G(a_i))$ by the elementary embedding, so $[Z_G(a_i):Z_H(a_i)]<\infty$ and hence $Z_H(a_i)=Z_G(a_i)$ by Lemma \ref{finite-index}. Thus $H$ is torally complete in $G$.
\end{proof}


%
%
%
%
%

\section{\texorpdfstring{$\Gamma$}{Gamma}-limit groups and \texorpdfstring{$\bbr$}{R}-trees}
For more information on limiting actions on $\bbr$-trees, see \cite{Bestvina02,GL17}.

\begin{dfn}
Let $G$ be a finitely generated group and let $\Gamma$ be a toral relatively hyperbolic group. A sequence $(f_n)_n$ in $\Hom(G,\Gamma)$ is \emph{stable} if for all $g\in G$ either (i) $g\in\ker(f_n)$ for all but finitely many $n$; or (ii) $g\notin\ker(f_n)$ for all but finitely many $n$. The \emph{stable kernel} of a stable sequence is
\[\stabker(f_n):=\{g\in G:g\in\ker(f_n)\text{ for all but finitely many }n\}.\]
A finitely generated group $L$ is a \emph{$\Gamma$-limit group} if there is a finitely generated group $G$ and a stable sequence $(f_n)_n$ in $\Hom(G,\Gamma)$ such that $L\cong G/\stabker(f_n)$.
\end{dfn}

\begin{dfn}
An \emph{$\bbr$-tree} is a geodesic metric space in which any two distinct points $a,b\in T$ are connected by a unique topological arc $[a,b]$. A non-empty subtree of $T$ is \emph{degenerate} if it is a single point, and otherwise it is \emph{non-degenerate}.

A \emph{tripod} in an $\bbr$-tree $T$ is a union of three arcs $[a,b]\cup[b,c]\cup[c,a]$ defined by three distinct points $a,b,c\in T$ which contains a \emph{branch point} $x\in T$ such that $T\setminus\{x\}$ consists of three connected components.

An action of a finitely generated group $G$ on an $\bbr$-tree $T$ is \emph{superstable} if, for any two non-degenerate arcs $[a,b]\subsetneq[c,d]$ such that $\Stab_G[c,d]\lneq\Stab_G[a,b]$, we have that $\Stab_G[c,d]=\{1\}$.
\end{dfn}

\begin{dfn}
Let $G$ be a group with a subgroup $H$ and let $\Gamma$ be a group with a fixed embedding $i:H\emb\Gamma$. A morphism $f:G\to\Gamma$ \emph{fixes} $H$ if $f|_H=i$. We also define
\[\Hom_H(G,\Gamma):=\{f\in\Hom(G,\Gamma):f|_H=i\}.\]
\end{dfn}

\begin{dfn}
Let $G=\lang S\rang$ be a finitely generated group (with $S$ finite) and let $\Gamma$ be a toral relatively hyperbolic group which acts by isometries on a pointed metric space $(X,d,*)$. The \emph{length} of a morphism $f:G\to\Gamma$ is defined to be
\[|f|:=\max_{s\in S}d(*,f(s)\cdot *).\]
\end{dfn}

\begin{dfn}
Let $G=\lang S\rang$ be a finitely generated group (with $S$ finite) which is freely indecomposable relative to a subgroup $H$, and let $\Gamma$ be a toral relatively hyperbolic group with a fixed embedding $i:H\emb\Gamma$. Let $(X,d)$ be the Cayley graph of $G$ with respect to $S$ with the word metric, and denote the \emph{ball in $G$ of radius $r$ (with respect to $S$)} by
\[B_G(r):=\{g\in G\subseteq X:d(1,g)\leq r\}\]

We say that a sequence $(f_n)_n$ in $\Hom(G,\Gamma)$ \emph{fixes $H$ in the limit} if for all $r\geq 1$ there exists $N_r$ such that $f_n$ coincides with $i$ on $B_G(r)\cap H$ for all $n\geq N_r$.
\end{dfn}

\begin{thm}[Compare {\cite[Theorem 6.5]{Groves09}}]\label{converge}
Let $G$ be a finitely generated group, let $\Gamma$ be a toral relatively hyperbolic group, and let $H$ be a non-abelian subgroup of $G$ with a fixed embedding $i:H\emb\Gamma$.

Suppose that $(f_n)_n$ is a sequence of distinct morphisms in $\Hom(G,\Gamma)$ fixing $H$ in the limit. Then there is a stable subsequence $(h_n)_n$ of $(f_n)_n$ and an $\bbr$-tree $T$ equipped with an isometric $G$-action with no global fixed point which have the following properties, where $K$ is the kernel of this action and $L=G/K$:
\begin{enumerate}[(1)]
\item
$H\cap K=\{1\}$, and in particular we may consider $H\leq L$.
\item
If $[a,b]$ is a non-degenerate arc in $T$, then $\Stab_L[a,b]$ is abelian.
\item
If $T$ is isometric to a real line, then $h_n(G)$ is free abelian for all but finitely many $n$.
\item
$\stabker(h_n)\leq K$.
\item
If $g\in G$ stabilizes a tripod in $T$, then $g\in\stabker(h_n)$.
\item
If $[a,b]\subseteq[c,d]$ are non-degenerate arcs in $T$ and $\Stab_L[c,d]\neq\{1\}$, then $\Stab_L[a,b]=\Stab_L[c,d]$, and in particular the action of $L$ on $T$ is superstable.
\item
$L$ is torsion free.
\item
$H$ fixes the basepoint of $T$.
\end{enumerate}
\end{thm}

\begin{proof}
The $\bbr$-tree $T$ and the isometric $G$-action are constructed as in \cite{Groves09}. In particular, parts (2)--(7) follow from \cite[Theorem 6.5]{Groves09}. It will remain to show that $H\cap K=\{1\}$, that $H$ fixes the basepoint, and that the is no global fixed point. Note that the non-existence of a global fixed point is not immediate from \cite[Lemma 6.2]{Groves09} because in order to keep $H$ fixed in the limit, we cannot use conjugation to ensure that the basepoint is centrally located. For more on Gromov-Hausdorff limits of pointed spaces, see \cite{Bestvina02}.

Let $(X,d_0,1)$ be the metric space constructed as in \cite[\S 4]{Groves09} in which $\Gamma$ isometrically embeds with basepoint corresponding to $1\in\Gamma$, and for each $n$ define the scaling factors $\delta_n:=|f_n|$. Define a sequence of pointed spaces $(X_n,d_n,*_n)$ with $X_n=X$, $d_n=d_0/\delta_n$, $*_n=1$, and a $G$-action $f_n(g)\cdot x$. A subsequence of these actions converge in the pointed Gromov-Hausdorff topology to a space $(X_\infty,d,*)$ upon which $G$ acts with no global fixed point. Let $C_\infty$ be the subspace of $X_\infty$ consisting of the union of all geodesic segments $[*,g\cdot *]$ and all flats containing geodesic triangles $\Delta(g\cdot *,g'\cdot *,g''\cdot *)$ for $g,g',g''\in G$. After projecting from the flats of $C_\infty$ as in \cite[\S 6.1]{Groves09} we obtain an $\bbr$-tree $T\subseteq C_\infty$ containing the basepoint $*$ and upon which $G$ acts isometrically with no global fixed point by \cite[Lemma 6.2]{Groves09}, and if $K$ is the kernel of this action we obtain an action of $L=G/K$ on $T$.

Let $g\in H$. Then
\[d_n(*_n,f_n(g)\cdot *_n)=d_n(*_n,i(g)\cdot 1)\text{ for }n\gg 1\text{, so}\]
\[d_n(*_n,f_n(g)\cdot *_n)=d_n(*_n,i(g)\cdot 1)=d_0(1,i(g)\cdot 1)/\delta_n\to 0=d(*,g\cdot *)\]
and hence $g$ fixes $*$ in $T$ for all $g\in H$. Thus $H$ fixes $*$.

It is clear that $H\cap\stabker(h_n)=\{1\}$ because $H$ is fixed in the limit, so suppose there exists $g\in K\setminus\stabker(h_n)$. But then $g$ cannot stabilize a tripod by (5), so $h\notin K$. Thus $H\cap K=\{1\}$ and hence we may consider $H\leq L$.

Because $H$ is non-abelian and fixed in the limit each $h_n(G)$ is non-abelian as well, so $T$ is not isometric to a real line by (3). Then if there existed a global fixed point $x\in T$ we would have $H\leq\Stab_L[*,x]$, which would contradict (2) because $H$ is non-abelian.

\end{proof}

\section{Modular automorphisms, bending, and shortening}


\begin{dfn}
Let $G$ be a finitely generated group with a splitting $\Lambda$ over abelian edge groups. If $Z=G_v$ is an abelian vertex subgroup of $G$ in $\Lambda$, let $I(Z)\leq Z$ be the subgroup generated by the incident edge stabilizers in $Z$. The \emph{incidental subgroup} $\bar{I}(Z)$ is the minimal direct factor of $Z$ containing $I(Z)$. We say that an abelian vertex subgroup $Z$ is \emph{incidental} if $\bar{I}(Z)=Z$.
\end{dfn}

\begin{dfn}
Let $G$ be a finitely generated group. A \emph{Dehn twist} is an automorphism of one of the following types:
\begin{itemize}
\item
If $G=A*_CB$ and $c\in Z(C)$, then define $\sigma\in\Aut(G)$ such that $\sigma(a)=a$ for $a\in A$ and $\sigma(b)=cbc^{-1}$ for $b\in B$.
\item
If $G=A*_C$ with stable letter $t$ and $c\in Z(C)$, then define $\sigma\in\Aut(G)$ such that $\sigma(a)=a$ for $a\in A$ and $\sigma(t)=tc$.
\end{itemize}

If $G$ has a splitting $\Lambda$ over abelian groups and $Z=G_v$ is an abelian vertex subgroup of $G$ in $\Lambda$, a \emph{generalized Dehn twist} is an automorphism $\sigma\in\Aut(G)$ which restricts to the identity on $\bar{I}(Z)$ and all other vertex groups.

The \emph{modular automorphism group} of $G$ is the subgroup $\Mod(G)$ of $\Aut(G)$ generated by Dehn twists, generalized Dehn twists, and inner automorphisms. If $H$ is a subgroup of $G$, the \emph{relative modular automorphism group} of $G$ relative to $H$ is the subgroup $\Mod_H(G)$ of $\Mod(G)$ consisting of modular automorphisms which restrict to the identity on $H$.
\end{dfn}

\begin{dfn}
Let $G$ be a finitely generated group which is freely indecomposable relative to a subgroup $H$, and let $\Gamma$ be a toral relatively hyperbolic group with a fixed embedding $i:H\emb\Gamma$. Two morphisms $f,f':G\to\Gamma$ fixing $H$ differ by a \emph{bending move} if either
\begin{enumerate}[(i)]
\item
there is a splitting $G=A*_CB$ relative to $H$ over an abelian subgroup $C$ with $H\leq A$ such that $f(C)$ is contained in a maximal non-cyclic abelian subgroup $Z\leq \Gamma$ and there exists $z\in Z$ such that $f'|_A=f|_A$ and $f'|_B=\iota_z\circ f_B$; or
\item
there is a splitting $G=A*_C=\lang A,t\rang$ relative to $H$ over an abelian subgroup $C$ with $H\leq A$ such that $f(C)$ is contained in a maximal non-cyclic abelian subgroup $Z\leq \Gamma$ and there exists $z\in Z$ such that $f'|_A=f|_A$ and $f'(t)=zt$.
\end{enumerate}
\end{dfn}

\begin{rmk}
These are the type (B2) bending moves of \cite[Definition 3.4]{Groves05}. We do not require the type (B1) bending moves for shortening because we restrict to splittings relative to a torally complete subgroup, and such a splitting can contain no non-incidental abelian vertex groups.
\end{rmk}

\begin{lem}
Let $G$ be a toral relatively hyperbolic group which is freely indecomposable relative to a non-abelian torally complete subgroup $H$, and let $\Lambda$ be the JSJ splitting of $G$ relative to $H$. Then $\Lambda$ contains no non-incidental abelian vertex groups.
\end{lem}

\begin{proof}
Cyclic vertex groups are trivially incidental because edge groups are nontrivial, so let $T$ be the Bass-Serre tree corresponding to this splitting, let $Z=G_v$ be a non-cyclic abelian vertex subgroup of $\Lambda$, and let $A=G_u$ be the vertex subgroup of $\Lambda$ which contains $H$. Then $A$ is non-abelian and $Z$ is conjugate into $H$ because $H$ is torally complete in $G$, so $gZg^{-1}\leq H\leq A$ for some $g\in G$. Choose lifts $\tilde{v},\tilde{u}\in V(T)$. Then $gZg^{-1}$ stabilizes $g\cdot\tilde{v}$ and $\tilde{u}$, so it stabilizes the arc $[g\cdot\tilde{v},\tilde{u}]$. Because $T$ is bipartite and $A$ is non-abelian this arc must have odd length, and in particular it must be an edge $\tilde{e}=(g\cdot\tilde{v},\tilde{u})\in E(T)$ stabilized by $gZg^{-1}$ because the action of $G$ on $T$ is 2-acylindrical. But then the edge $g^{-1}\cdot\tilde{e}=(\tilde{v},g^{-1}\cdot\tilde{u})$ is stabilized by $Z$, so $Z\leq\bar{I}(Z)\leq Z$ and hence $Z=\bar{I}(Z)$ is incidental.
\end{proof}

\begin{dfn}
Let $G$ be a finitely generated group which is freely indecomposable relative to a subgroup $H$, and let $\Gamma$ be a toral relatively hyperbolic group with a fixed embedding $i:H\emb\Gamma$ such that $\Gamma$ acts by isometries on a pointed metric space $(X,d,*)$. Define an equivalence relation on $\Hom_H(G,\Gamma)$ generated by the relation $f\sim f'$ if either (i) $f'=f\circ\alpha$ for some $\alpha\in\Mod_H(G)$ or (ii) $f'$ differs from $f$ by a bending move. Note that any $f':G\to\Gamma$ such that $f\sim f'$ must fix $H$, so a morphism $f:G\to\Gamma$ fixing $H$ is \emph{short} if, for all $f':G\to\Gamma$ such that $f'\sim f$, $|f|\leq|f'|$.
\end{dfn}

\begin{thm}[Shortening Argument]\label{shortening}
Let $G=\lang S\rang$ be a finitely generated group (with $S$ finite) which is freely indecomposable relative to a non-abelian subgroup $H$, and let $\Gamma$ be a toral relatively hyperbolic group with a fixed embedding $i:H\emb\Gamma$ such that $i(H)$ is torally complete in $\Gamma$.

Suppose that $(f_n)_n$ is a sequence of distinct morphisms in $\Hom(G,\Gamma)$ fixing $H$ in the limit. If $(f_n)_n$ converges to a faithful isometric $G$-action on an $\bbr$-tree $T$, then all but finitely many $f_n$ are not short.
\end{thm}

\begin{proof}
We will use the Rips machine to analyze the action of $G$ on $T$. Because $H$ is fixed in the limit, $H$ fixes the basepoint point $*\in T$ by Theorem \ref{converge}, and further we can shorten using elements of $\Mod_H(G)$ and bending moves. Because $G$ is freely indecomposable relative to $H$ there are no Levitt components, and there are no axial components because non-cyclic abelian subgroups in the image of $f_n$ are conjugate into the elliptic subgroup $i(H)$. Thus $T$ only consists of IET and discrete components, so we can shorten the segments $[*,s\cdot *]$ for $s\in S$ as in \cite[Theorem 3.7]{Groves05}.
\end{proof}

\section{Applications of the shortening argument}
\label{sect5}

\subsection{Shortening quotients}

\begin{dfn}
Let $A=\lang S\rang$ be a finitely generated group (with $S$ finite) which is freely indecomposable relative to a subgroup $H$, and let $G$ be a toral relatively hyperbolic group with a fixed embedding $i:H\emb G$.

A \emph{$G$-limit quotient of $A$ (relative to $H$)} is a quotient $L=A/K$, where $K=\stabker(f_n)$ is the stable kernel of a stable sequence $(f_n)_n$ of non-injective morphisms in $\Hom(A,G)$ (which fix $H$ in the limit). Note that $L$ is a $G$-limit group which contains a copy of $H$ in the relative case because we will have $H\cap K=\{1\}$. If each $f_n$ is short (relative to $H$), then $L$ is a \emph{$G$-shortening quotient of $A$ (relative to $H$)}.

Define an order $\geq$ on the set of (relative) $G$-limit quotients by setting $(q:A\surj L)\geq(q':A\surj L')$, or simply $L\geq L'$, if there exists a morphism $f:L\to L'$ such that $q'=f\circ q$. The quotients $L$ and $L'$ are said to be equivalent if the map $f$ is an isomorphism. This defines an equivalence relation on the set of (relative) $G$-limit quotients. Additionally, this defines an order and an equivalence relation on the set of (relative) $G$-shortening quotients.
\end{dfn}

\begin{rmk}
\begin{enumerate}[(1)]
\item
It is not standard to assume that the morphisms $f_n$ are non-injective.
\item
Any sequence $L_1\geq L_2\geq L_3\geq\cdots$ of (relative) $G$-limit quotients terminates by \cite[Theorem 5.2]{Groves05}.
\item
If $q:A\surj L$ is a (relative) $G$-limit quotient corresponding to a stable sequence $(f_n)_n$ in $\Hom_H(A,G)$, then all but finitely many $f_n$ factor through $q$ by \cite[Theorem 5.6]{Groves05}.
\item
Given any sequence $L_1\leq L_2\leq L_3\leq\cdots$ of (relative) $G$-limit quotients of $A$, there exists a $G$-limit quotient $L$ of $A$ such that $L\geq L_i$ for all $i$ by \cite[Proposition 5.13]{Groves05}. In the relative case, if $(f_n)_n$ is the stable sequence in $\Hom_H(A,G)$ corresponding to $L$, then $H$ must be fixed in the limit by (3) because $L$ contains a copy of $H$ as noted above. Thus $L$ is a relative $G$-limit quotient. We note that \cite[Proposition 5.13]{Groves05} applies and that $L$ is a $G$-limit quotient in our sense because $L$ is defined by a sequence of non-injective morphisms in the original proof.
\end{enumerate}
\end{rmk}

\begin{lem}\label{lemsq1}
Let $A$ be a non-abelian toral relatively hyperbolic group which is freely indecomposable relative to a non-abelian subgroup $H$, and let $G$ be a toral relatively hyperbolic group with a fixed embedding $i:H\emb G$. Then $G$-shortening quotients of $A$ relative to $H$ are proper quotients.
\end{lem}

\begin{proof}
Let $L=A/K$ be a $G$-shortening quotient relative to $H$, where $K=\stabker(f_n)$ for a stable sequence $(f_n)_n$ of non-injective short morphisms in $\Hom(A,G)$ which fix $H$ in the limit. If $L$ is not proper, then $K=\{1\}$ is trivial, so the morphisms are distinct and this sequence converges to a faithful $A$-action on an $\bbr$-tree (possibly after passing to a subsequence) by Theorem \ref{converge}. This action has no global fixed point, trivial tripod stabilizers, and $H$ acting elliptically, so by Theorem \ref{shortening} not all $f_n$ could have been short, which is a contradiction.
\end{proof}

\begin{lem}\label{lemsq2}
Let $A=\lang S\rang$ be a finitely generated group (with $S$ finite) which is freely indecomposable relative to a subgroup $H$, and let $G$ be a toral relatively hyperbolic group with a fixed embedding $i:H\emb G$. There are only finitely many proper maximal $G$-shortening quotients of $A$ (relative to $H$) up to equivalence.
\end{lem}

We omit the proof of the preceding proposition as it is nearly identical to that of \cite[Lemma 6.2]{Groves05}, except that we only consider morphisms which fix a subgroup. As above, the proof of \cite[Lemma 6.2]{Groves05} applies in our context because the morphisms used are non-injective.

\begin{dfn}\label{star}
Let $A$ be a toral relatively hyperbolic group which is freely indecomposable relative to a subgroup $H$, let $G$ be a toral relatively hyperbolic group with a fixed embedding $i:H\emb G$, and let $\Lambda$ be the JSJ splitting of $A$ relative to $H$. We say that a map $f:A\to G$ satisfies $*_{H,A,G}$ if $f$ fixes $H$ and, for all $e\in E(\Lambda)$,
\[\rank(Z_A(A_e))=\rank(Z_G(f(A_e))).\]
\end{dfn}


\begin{lem}\label{nobend}
Let $A$ be a toral relatively hyperbolic group which is freely indecomposable relative to a subgroup $H$, let $G$ be a toral relatively hyperbolic group with a fixed embedding $i:H\emb G$. Suppose that $f:A\to G$ satisfies $*_{H,A,G}$ for some JSJ splitting $\Lambda$ of $A$ relative to $H$. If either (i) $H$ is torally complete in $A$ and the image $i(H)$ is torally complete in $G$; or (ii) $H=\{1\}$ and $A$ contains no non-cyclic abelian subgroups (i.e., $A$ is freely indecomposable and hyperbolic), then for any $f':A\to G$ such that $f'$ differs from $f$ by a bending move there exists $\sigma\in\Mod_H(A)$ such that $f'=f\circ\sigma$, i.e., $f$ admits no non-trivial bending moves.
\end{lem}

\begin{proof}
The result is clear in case (ii) because the parabolic subgroups of $G$ are the non-cyclic abelian subgroups of $G$, so the condition $*_{H,A,G}$ ensures that $Z_G(f(A_e))$ is cyclic and hence not conjugate into a parabolic subgroup, so suppose that $H$ is torally complete in $A$ and $i(H)$ is torally complete in $G$.

Any $A_e$ which is not conjugate in $A$ into $H$ must have a cyclic centralizer, so because $Z_G(f(A_e))$ is also cyclic the image of $A_e$ is not contained in a parabolic subgroup of $G$ and hence there are no bending moves in such edges.

Suppose that $A_e$ is conjugate into $H$ by $a\in A$. Let $M=Z_A(A_e)$ and $N=Z_G(f(A_e))$. Then $aMa^{-1}$ is maximal abelian in $H$, so $i(aMa^{-1})$ is maximal abelian in $i(H)$ and hence also in $G$ because $i(H)$ is torally complete. We also have that $f(a)Nf(a)^{-1}$ is maximal abelian in $G$, so because
\[f(a)f(M)f(a)^{-1}=f(aMa^{-1})=i(aMa^{-1})\leq i(H)\]
is a finite index subgroup of $f(a)Nf(a)^{-1}$ by $*_{H,A,G}$ we have that $f(a)Nf(a)^{-1}$ is maximal abelian in $i(H)$. Then $N=f(M)$, so for any $n\in N$ there exists $m\in M$ such that $n=f(m)$. Thus any bending in the edge $A_e$ by the element $n$ can be realized as precomposing by a Dehn twist by the element $m$.
\end{proof}

\begin{lem}
Let $G$ be a toral relatively hyperbolic group and let $A$ be a toral relatively hyperbolic retract of $G$ which is freely indecomposable relative to a subgroup $H$. If either (i) $H$ is torally complete in $G$; or (ii) $H=\{1\}$ and no element of $A$ is conjugate into a non-cyclic abelian subgroup of $G$, then the usual embedding $A\emb G$ satisfies $*_{H,A,G}$.
\end{lem}

\begin{proof}
Again, the result is clear in case (ii), so suppose that $H$ is torally complete in $G$.

It is clear that the usual embedding $i:A\emb G$ fixes $H$, so let $\Lambda$ be the JSJ splitting of $A$ relative to $H$ and let $r:G\surj A$ be the retraction of $G$ onto $A$. Let $e\in E(\Lambda)$, $M=Z_A(A_e)$, and $N=Z_G(i(A_e))=Z_G(A_e)$. 

Suppose that $A_e$ is conjugate into $H$ by $a\in A$. Then $aMa^{-1}\leq H$ is maximal abelian in both $G$ and $H$ because $H$ is torally complete in $G$, and further $aMa^{-1}\leq aNa^{-1}$. Then $aMa^{-1}=aNa^{-1}$ and hence $M=N$, so $\rank(M)=\rank(N)$.

Any $A_e$ which is not conjugate into $H$ in $A$ is also not conjugate into $H$ in $G$ by the retraction, so both $A_e$ and $M$ must be cyclic. Suppose that $\rank(N)>\rank(M)=1$. Then there exists $g\in G$ such that $gNg^{-1}\leq H$. But then
\[r(g)A_er(g)\leq r(g)Mr(g)^{-1}=r(gMg^{-1})\leq r(gNg^{-1})=gNg^{-1}\leq H,\]
which contradicts the assumption that $A_e$ is not conjugate into $H$.

Thus $A\emb G$ satisfies $*_{H,A,G}$.

\end{proof}

\begin{cor}\label{cor58}
Let $G$ be a toral relatively hyperbolic group and let $A$ be a toral relatively hyperbolic retract of $G$ which is freely indecomposable relative to a subgroup $H$. If either (i) $H$ is torally complete in $G$; or (ii) $H=\{1\}$ and no element of $A$ is conjugate into a non-cyclic abelian subgroup of $G$, then the usual embedding $A\emb G$ admits no non-trivial bending moves.
\end{cor}

\begin{lem}\label{starinj}
Let $A$ be a toral relatively hyperbolic group which is freely indecomposable relative to a subgroup $H$, let $G$ be a toral relatively hyperbolic group with a fixed embedding $i:H\emb G$, and let $f,f':A\to G$ be morphisms fixing $H$ such that $f\sim f'$. If either (i) $H$ is torally complete in $A$ and the image $i(H)$ is torally complete in $G$; or (ii) $H=\{1\}$ and $A$ contains no non-cyclic abelian subgroups, then
\begin{enumerate}[(1)]
\item
$f$ satisfies $*_{H,A,G}$ if and only if $f'$ satisfies $*_{H,A,G}$; and
\item
If $f$ satisfies $*_{H,A,G}$, then $f$ is injective if and only if $f'$ is injective.
\end{enumerate}
\end{lem}

\begin{proof}
Suppose that $f$ satisfies $*_{H,A,G}$. Then $f$ admits no nontrivial bending moves by Lemma \ref{nobend}, so $f'=f\circ\sigma$ for some $\sigma\in\Mod_H(A)$, and thus $f'$ satisfies $*_{H,A,G}$ because $\sigma$ clearly satisfies $*_{H,A,A}$. Similarly we find that $f$ is injective if and only if $f'$ is injective.
\end{proof}

\begin{prop}\label{sq1}
Let $A$ be a toral relatively hyperbolic group which is freely indecomposable relative to a subgroup $H$, and let $G$ be a toral relatively hyperbolic group with a fixed embedding $i:H\emb G$. If $H$ is torally complete in $A$ and the image $i(H)$ is torally complete in $G$, then there exists a finite set of proper quotients of $A$ and a finite subset $H_0\subseteq H$ such that $\lang H_0\rang$ is torally complete in $A$, $i(\lang H_0\rang)$ is torally complete in $G$, and, for any non-injective morphism $f:A\to G$ fixing $H_0$ and satisfying $*_{\lang H_0\rang,A,G}$, there exists $\sigma\in\Mod_H(A)$ making $f\circ\sigma$ factor through one of these proper quotients.
\end{prop}

\begin{proof}
Suppose that no such subset $H_0$ exists and let $S$ be the finite set consisting of generators from representatives of each of the conjugacy classes of non-cyclic abelian subgroups of $H$. Then there exists a sequence of integers $(r_n)_n$ such that $r_n\to\infty$ and a stable sequence $(f_n)_n$ of non-injective morphisms in $\Hom(A,G)$ such that each $f_n$ fixes the finite subset $B_n:=(B_A(r_n)\cup S)\cap H$, satisfies $*_{\lang B_n\rang,A,G}$ and does not not factor through any of the maximal $G$-shortening quotients relative to $H$, and such that $\lang B_n\rang$ is torally complete in $A$ and $i(\lang B_n\rang)$ is torally complete in $G$.

Then $H$ is fixed in the limit, and furthermore we may assume each $f_n$ is short relative to $\lang B_n\rang$ because any morphism $f'\sim f_n$ is also non-injective by Lemma \ref{starinj}. Let $q:A\surj L$ be the corresponding $G$-shortening quotient of $A$ relative to $H$. Then there exists some maximal relative $G$-shortening quotient $M$ such that $M\geq L$ by Lemma \ref{lemsq2}. However, all but finitely many $f_n$ factor through $q$ and hence factor through $A\surj M$, which is a contradictiton.
\end{proof}

\begin{prop}\label{sq2}
Let $A$ be a freely indecomposable, torsion-free hyperbolic group, and let $G$ be a toral relatively hyperbolic group. There exists a finite set of proper quotients of $A$ such that, for any non-injective morphism $f:A\to G$ satisfying $*_{\{1\},A,G}$, there exists $\sigma\in\Mod(A)$ making $f\circ\sigma$ factor through one of these proper quotients.
\end{prop}

\begin{proof}
Suppose the result does not hold. Then there exists a stable sequence $(f_n)_n$ of non-injective morphisms in $\Hom(A,G)$ such that each $f_n$ satisfies $*_{\lang 1\rang,A,G}$, and does not not factor through any of the maximal $G$-shortening quotients relative to $\{1\}$.

Then $H$ is fixed in the limit, and furthermore we may assume each $f_n$ is short relative to $\lang B_n\rang$ because any morphism $f'\sim f_n$ is also non-injective by Lemma \ref{starinj}. Let $q:A\surj L$ be the corresponding $G$-shortening quotient of $A$ relative to $H$. Then there exists some maximal relative $G$-shortening quotient $M$ such that $M\geq L$ by Lemma \ref{lemsq2}. However, all but finitely many $f_n$ factor through $q$ and hence factor through $A\surj M$, which is a contradictiton.
\end{proof}

\begin{prop}\label{hypembs}
Let $A$ be a freely indecomposable, torsion-free hyperbolic group and let $G$ be a toral relatively hyperbolic group. There exists a finite set of embeddings $A\emb G$ such that, for any embedding $i:A\emb G$ satisfying $*_{\{1\},A,G}$, there exists $\sigma\in\Mod(A)$ and $g\in G$ so that $\iota_g\circ i\circ\sigma$ is equal to one of the embeddings in the finite set.
\end{prop}

\begin{proof}
Suppose no such set of embeddings exists. Then there exists a stable sequence $(j_n)_n$ in $\Hom(A,G)$ of distinct, non-equivalent embeddings satisfying $*_{\{1\},A,G}$ with $\stabker(j_n)=\{1\}$. We may also assume that each $j_n$ is short relative to $\{1\}$ because any morphism $f'\sim j_n$ is also an embedding by Lemma \ref{starinj}. Furthermore, no $j_n$ will admit non-trivial bending moves by Corollary \ref{cor58}. The sequence will converge to a faithful isometric $G$-action on an $\bbr$-tree $T$, so not all $j_n$ could have been short by \cite[Theorem 3.7]{Groves05}.
\end{proof}

\subsection{Not the co-Hopf property}

\begin{dfn}
A group $G$ is \emph{co-Hopf} (or has the co-Hopf property) if any injective morphism $G\emb G$ is an isomorphism. Similarly, if $H$ is a subgroup of $G$, then $G$ is \emph{co-Hopf relative to $H$} if any injective morphism $G\emb G$ which restricts to the identity on $H$ is an isomorphism.
\end{dfn}

\begin{rmk}
If $G$ has a splitting $\Lambda$ over abelian groups which contains a non-incidental abelian vertex subgroup, then $G$ is not co-Hopf: Suppose $Z=\bar{I}(Z)\times B$ is a non-incidental abelian vertex subgroup. We can define a monomorphism $G\emb G$ which is not an epimorphism by mapping $B$ into a proper finite-index subgroup of itself and restricting to the identity map elsewhere.
\end{rmk}

Unfortunately we were unable to prove the more general relative co-Hopf property we had hoped for. The co-Hopf property for torsion-free hyperbolic groups was used in \cite{Perin11} to rule out the existence of injective preretractions, but we found a workaround using the property $*_{H,A,G}$. We do not know if there exist injective morphisms $A\emb G$ which fix $H$ but do not satisfy $*_{H,A,G}$, but it is clear that the usual embedding $i:A\emb G$ satisfies $*_{H,A,G}$, and hence so would $i\circ\alpha$ for any $\alpha\in\Mod_H(G)$. Then any morphism which is $\Lambda$-related to $i$ also satisfies $*_{H,A,G}$, and hence preretractions satisfy $*_{H,A,G}$.

\begin{lem}\label{emb2}
Let $G$ be a toral relatively hyperbolic group, let $H$ be a non-abelian, torally complete subgroup of $G$, and let $A$ be a toral relatively hyperbolic retract of $G$ which contains and is freely indecomposable relative to $H$. There exists a finite set $H_0\subseteq H$ such that any embedding $A\emb G$ satisfying $*_{\lang H_0\rang,A,G}$ fixes $H$.
\end{lem}

\begin{proof}
Suppose that no such subset $H_0$ exists and let $S$ be the finite set consisting of generators from representatives of each of the conjugacy classes of non-cyclic abelian subgroups of $H$. Then there exists a sequence of integers $(r_n)_n$ such that $r_n\to\infty$ and a sequence $(i_n)_n$ of distinct embeddings in $\Hom(A,G)$ such that each $i_n$ fixes the finite subset $B_n:=(B_A(r_n)\cup S)\cap H$ but not $H$ and satisfies $*_{\lang B_n\rang,A,G}$, and such that $\lang B_n\rang$ is torally complete in $G$. Then $H$ is fixed in the limit, and furthermore we may assume each $i_n$ is short relative to $\lang B_n\rang$ because any morphism $f'\sim i_n$ is also an embedding. Then $\stabker(i_n)=\{1\}$ is trivial, so this sequence converges to a faithful action on an $\bbr$-tree (possibly after passing to a subsequence) by Theorem \ref{converge}, so by Theorem \ref{shortening} not all $i_n$ could have been short, which is a contradiction.
\end{proof}

\begin{prop}\label{prop5.16}
Let $H$ be a proper torally complete subgroup of a toral relatively hyperbolic group $A$ which is freely indecomposable relative to $H$, and let $\Lambda$ be the JSJ splitting of $A$ relative to $H$. Then there exists a finite subset $H_0\subseteq H$ and a finite set $K_0\subseteq A\setminus\{1\}$ such that for any morphism $f:A\to H$ which fixes $H_0$ and satisfies $*_{\lang H_0\rang,A,H}$, there exists $\sigma\in\Mod_H(A)$ and $k\in K_0$ such that $f\circ\sigma(k)=1$.
\end{prop}

\begin{proof}
Choose $H_0$ to satisfy the conditions of Proposition \ref{sq1} and Lemma \ref{emb2} for morphisms $A\to A$ satisfying $*_{\lang H_0\rang,A,A}$ and choose $K_0$ to consist of a single non-trivial element from the kernel of each of the proper quotients of Proposition \ref{sq1}. If $f$ is non-injective the result follows immediately by Proposition \ref{sq1}, so suppose that there exists an injective $f:A\emb H$ as above and let $i:H\emb A$ be the usual embedding. Then $g:=i\circ f:A\emb A$ fixes $H_0$ and satisfies $*_{\lang H_0\rang,A,A}$. Then $g$ and hence $f$ must both fix $H$ by Lemma \ref{emb2}, so $f$ is surjective and could not have been injective because $H$ is a proper subgroup fixed by $f$.
\end{proof}

\section{Preretractions from elementary embeddings}
\label{sec6}
\begin{lem}[{Compare \cite[Lemma 5.18]{Perin11}}]\label{lambdarel}
Let $A=\lang\ua\rang$ be a finitely generated group with a JSJ-like splitting $\Lambda$ such that all vertex and edge groups are finitely generated. There exists a first-order formula $\rho_\Lambda(\ux,\uy)$ such that for any morphisms $f,f':A\to G$ given by $\ua\mapsto\us$ and $\ua\mapsto\ut$, $f$ and $f'$ are $\Lambda$-related if and only if $G\models\rho_\Lambda(\us,\ut)$.
\end{lem}

\begin{proof}
Given a tuple $\ux=(x_1,\ldots,x_n)$, define
\[\alpha(\ux)\ :\ \bigwedge_{1\leq i<j\leq n} [x_i,x_j]=1.\]
Then $G\models\alpha(\ux)$ if and only if the subgroup $\lang\ux\rang$ is abelian.

For each $v\in V(\Lambda)$ and each $e\in E:=E(\Lambda)$ choose finite generating sets $S_v:=\{\ub=\ub(\ua)\}$ and $S_e:=\{\uc=\uc(\ua)\}$ so that $A_v=\lang S_v\rang$ and $A_e=\lang S_e\rang$. Partition $V(\Lambda)=U\amalg V$ so that $V$ consists of the non-exceptional surface-type vertices of $\Lambda$ and $U$ consists of all other vertices.

For each $u\in U$ and $e\in E$ define
\begin{align*}
\rho_u(\ux,\uy)\ &:\ \exists z\bigwedge_{b\in S_u} b(\ux)=zb(\uy)z^{-1},\\
\rho_e(\ux,\uy)\ &:\ \exists z\bigwedge_{c\in S_e} c(\ux)=zc(\uy)z^{-1},
\end{align*}
and for each $v\in V$ let $S_v=\{\ub=\ub(\ua)\}$ and define
\[\rho_v(\ux,\uy)\ :\ \neg\alpha(\ub(\ux))\implies\neg\alpha(\ub(\uy)).\]
Then the result follows if we define
\[\rho_\Lambda(\ux,\uy)\ :\ \left(\bigwedge_{u\in U}\rho_u(\ux,\uy)\right)\wedge\left(\bigwedge_{v\in V}\rho_v(\ux,\uy)\right)\wedge\left(\bigwedge_{e\in E}\rho_e(\ux,\uy)\right).\]
\end{proof}

\begin{lem}\label{epsilon}
Let $A=\lang\ua\rang$ be a finitely generated group with the JSJ-like splitting $\Lambda$ and let $G$ be a torsion-free commutative-transitive group with all abelian subgroups finitely generated. There exists a first-order formula $\epsilon_\Lambda(\ux)$ such that for any morphism $f:A\to G$ given by $\ua\mapsto\us$ with $f(Z_A(A_e))\neq\{1\}$ for all $e\in E(\Lambda)$, $f$ satisfies
\[\rank(Z_A(A_e))=\rank(Z_G(f(A_e)))\]
for all $e\in E(\Lambda)$ if and only if $G\models\epsilon_\Lambda(\us)$.
\end{lem}

\begin{proof}
For each $e\in E(\Lambda)$ let $r_e=\rank(Z_A(A_e))$ and choose a basis
$\{b_{e,i}=b_{e,i}(\ua)\}_{i=1}^{r_e}$ for $Z_A(A_e)$. Since $E(\Lambda)$ is finite by Lemma \ref{ranks} we can define
\begin{align*}
\epsilon_e(\ux)\ &:\ \left(\bigvee_{i=1}^{r_e}b_{e,i}(\ux)\neq 1\right)\wedge\left(\bigwedge_{i=1}^{r_e}\left((b_{e,i}(\ux)\neq 1)\implies\cz_{r_e}(b_{e,i}(\ux))\right)\right),\\
\epsilon_\Lambda(\ux)\ &:\ \bigwedge_{e\in E(\Lambda)}\epsilon_e(\ux).
\end{align*}
Then $f(b_{e,i})=b_{e,i}(\us)$, so $G\models\epsilon_e(\us)$ if and only if $f(Z_A(A_e))\neq\{1\}$ and
\[\rank(Z_G(f(A_e)))=\rank(Z_G(f(b_{e,i})))=r_e=\rank(Z_A(A_e)),\]
so the result follows.
\end{proof}

\mainpropa

\begin{proof}
Let $\Lambda$ be the JSJ splitting of $A$ relative to $H$, let $A=\lang\ua|\Sigma(\ua)=1\rang$ be a finite presentation for $A$, and let $H_0=\{h_1,\ldots,h_n\}\subseteq H$ and $K_0=\{k_1,\ldots,k_m\}\subseteq A$ as in Proposition \ref{prop5.16} with $h_i=h_i(\ua)$ and $k_j=k_j(\ua)$, where we are concerned with maps $A\to H$ which satisfy $*_{\lang H_0\rang,A,H}$.

Define the first-order formula
\begin{align*}
\ph(z_1,\ldots,z_n)\ :\ 
	&\forall\ux\exists \uy\ \left((\Sigma(\ux)=1)\wedge\epsilon_{\Lambda}(\ux)\wedge\left(\bigwedge_{i=1}^n z_i=h_i(\ux)\right)\right)\\
	&\implies \left((\Sigma(\uy)=1)\wedge\rho_{\Lambda}(\ux,\uy)\wedge\left(\bigvee_{i=1}^m k_i(\uy)=1\right)\right)
\end{align*}

Then by Lemmas \ref{lambdarel} and \ref{epsilon}, $\ph(h_1,\ldots,h_n)$ can be interpreted as meaning that, for any morphism $f:A\to H$ given by $\ua\mapsto\ux$ which satisfies $*_{\lang H_0\rang,A,H}$, there exists a $\Lambda$-related morphism $f':A\to H$ given by $\ua\mapsto\uy$ such that $f'(k_i)=1$ for some $i$. Then $H\models\ph(h_1,\ldots,h_n)$ by Proposition \ref{prop5.16}, so because $H\elemb G$ we have $G\models\ph(h_1,\ldots,h_n)$. Interpreted over $G$, this implies that because the inclusion $A\emb G$ satisfies $*_{\lang H_0\rang,A,G}$, there exists a non-injective morphism $p:A\to G$ which is $\Lambda$-related to the inclusion $A\emb G$. Then $p:A\sqto G$ is a non-injective preretraction with respect to $\Lambda$.
\end{proof}

\mainpropb

\begin{proof}
Let $\Lambda$ be the JSJ splitting of $C$, let $C=\lang\uc|\Sigma(\uc)=1\rang$ be a finite presentation for $C$, let $\{\uu\} = \{u_1,\ldots,u_n\}\subseteq\Hom(C,H)$ be the finite set of embeddings as in Proposition \ref{hypembs}, and choose $K_0=\{k_1,\ldots,k_m\}\subseteq C$ to consist of a single non-trivial element from the kernel of each of the proper quotients of Proposition \ref{sq2} with $k_j=k_j(\uc)$, where in both cases we are concerned with maps $C\to H$ which satisfy $*_{\{1\},C,H}$.

Define the first-order formula with constants $u_i(\uc)\in H$
\begin{align*}
\ph(\ux)\ :\ 
	&((\Sigma(\ux)=1)\wedge \epsilon_\Lambda(\ux))\\
	&\wedge\forall\uy\ \left(((\Sigma(\uy)=1)\wedge \rho_\Lambda(\ux,\uy))\implies\left(\bigwedge_{i=1}^n\uy\neq u_i(\uc)\right)\right)
\end{align*}

Then $H\models\ph(\ux)$ if and only if $\uc\mapsto\ux$ defines a morphism $f:C\to H$ which satisfies
\[\rank(Z_A(A_e))=\rank(Z_H(f(A_e)))\]
for all $e\in E(\Lambda)$ and is not $\Lambda$-related to any of the embeddings $u_i$. By Proposition \ref{hypembs}, this is sufficient to ensure that morphisms satisfying $*_{\{1\},C,H}$ are non-injective, i.e., if $H\models\ph(\ux)$ then $\uc\mapsto\ux$ defines a non-injective morphism $C\to H$.

Define the first-order sentence with constants in $H$
\begin{align*}
\psi\ :\ 
	&\forall \ux\ \ph(\ux)\\
	&\implies\exists\uy\ \left((\Sigma(\uy)=1)\wedge\rho_{\Lambda}(\ux,\uy)\wedge\left(\bigvee_{i=1}^m k_i(\uy)=1\right)\right)
\end{align*}

Then $H\models\psi$ because any non-injective morphism $C\to H$ given by $\uc\mapsto\ux$ and satisfying $*_{\{1\},C,H}$ is $\Lambda$-related to a morphism $\uc\mapsto\uy$ with some $k_i$ in its kernel by Proposition \ref{sq2}. By the elementary embedding we have $G\models\psi$.

Because $C$ is neither cyclic nor a non-exceptional surface group, $\Lambda$ has at least one non-surface-type vertex subgroup $B$. Any $f:C\to G$ which is $\Lambda$-related to the embedding $C\emb G$ given by $\uc\mapsto\uc$ restricts to conjugation on $B$. Then $f(B)$ cannot be in $H$ by hypothesis, so the tuple $f(\uc)\neq u_i(\uc)$ for all $i$ and hence $G\models\ph(\uc)$. Then there exists a non-injective morphism $p:C\to G$ which is $\Lambda$-related to the inclusion $C\emb G$, so $p:C\sqto G$ is a non-injective preretraction with respect to $\Lambda$.
\end{proof}

\bibliographystyle{plain}
\bibliography{Elementary_Embeddings.bib}

\noindent
---\\
Christopher Perez\\
Loyola University New Orleans\\
\href{mailto:caperez@loyno.edu}{\texttt{caperez@loyno.edu}}

\end{document}